\documentclass[10pt]{article}
\usepackage[T1]{fontenc}
\usepackage[utf8]{inputenc}
\usepackage[english]{babel}
\usepackage{authblk}
\usepackage{amsmath,amssymb,mathrsfs,amsthm}
\usepackage{titlesec}
\usepackage[filecolor=green]{hyperref}
\usepackage[toc,page]{appendix}
\usepackage{color}
\usepackage[a4paper]{geometry}
\usepackage{tikz, tikz-cd}
\usepackage{enumitem}
\usepackage{fourier}
\setlist[itemize]{itemsep=.2cm, leftmargin=.75cm}
\setlist[description]{itemsep=.25cm, leftmargin=0cm, font=\normalfont\underline} 
\usepackage{stmaryrd}

\newtheorem{theorem}{Theorem}[section]
\newtheorem{lemma}[theorem]{Lemma}
\newtheorem{proposition}[theorem]{Proposition}

\theoremstyle{definition}
\newtheorem{definition}[theorem]{Definition}

\newcommand\ie{{\em i.e.}}
\newcommand\andd{\text{and}}
\newcommand\andquad{\quad\andd\quad}
\newcommand\andqquad{\qquad\text{and}\qquad}

\newcommand\M{{\bf M}}
\newcommand\NN{\mathbb{N}}
\newcommand\kk{\mathbb{k}}
\newcommand\X{x_1,\cdots,x_d}
\newcommand\monX{[x_1,\cdots, x_d]}
\newcommand\kkX{\kk[[x_1,\cdots, x_d]]}
\newcommand\gen{\vv{s_1},\cdots,\vv{s_q}}
\newcommand\vv[1]{\overrightarrow{#1}}
\newcommand\vve[1]{\vv{e}_{#1}}
\newcommand\Span[1]{\left\langle#1\right\rangle}
\DeclareMathOperator\supp{supp}
\DeclareMathOperator\mon{mon}
\DeclareMathOperator\syz{syz}
\DeclareMathOperator\lcm{lcm}

\newcommand\monOrd{\prec}
\newcommand\monOrdLeq{\preceq}
\newcommand\Scal{{\cal S}}
\newcommand\SchOrd{\monOrd_{\Scal}}
\DeclareMathOperator\lm{lm}
\DeclareMathOperator\lc{lc}

\begin{document}

\title{Schreyer resolution of modules\\[.2cm]
  over formal power series}
\author{Cyrille Chenavier\thanks{Corresponding author:
    \texttt{cyrille.chenavier@unilim.fr}}}
\author{Thomas Cluzeau}
\author{Adya Musson-Leymarie} 
\affil{Univ. Limoges, CNRS, XLIM, UMR 7252, F-87000 Limoges}
\date{}

\maketitle
      
\begin{abstract}
  Standard bases of modules over algebras of formal power series play the same 
  role as Gröbner bases of modules over polynomial algebras. In this article, we 
  first prove the analogue of the diamond lemma for modules over formal power
  series; it characterises standard bases in terms of unique remainders and
  standard representations. Then, using standard representations, we provide a
  method to construct a standard basis of the module of syzygies of a standard
  basis. This construction can be applied inductively to obtain a free
  resolution, similar to the Schreyer resolution, for finitely presented modules
  over formal power series.  
\end{abstract}
\noindent
\begin{small}\textbf{Keywords:} formal power series, standard bases, Schreyer
  resolution.
\end{small}

\tableofcontents

\newpage

\section{Introduction}

In computer algebra, effective computations often require a well-suited notion
of rewriting systems. These are mathematical models that formalise calculations 
and provide methods for determining the representatives of equivalence classes
in a given algebraic context. Starting from a presentation of an algebraic
object given by generators and relations, the approach uses the relations to
simplify the elements of the free object by replacing them with equivalent
elements. If the rewriting system satisfies two fundamental properties called
termination and confluence, meaning that the computations yield deterministic
results, then it can be applied to various problems of algorithmic and abstract
algebra: for example, solving decision problems, computing linear or homotopy
bases, constructing free resolutions, and proving homological
properties~\cite{An86, DK10, GHM19, GM12, Kob90, Squ87}.  
\smallskip

In the context of commutative algebras, the classical tools for constructing
terminating and confluent rewriting systems for algebraic objects given by 
generators and relations are Gröbner bases. They require to work with a monomial
order, so that the simplification rules consist in replacing each leading
monomial appearing in the relations by the corresponding remainder; the
properties of a monomial order ensure termination of this process. Moreover,
proving confluence reduces to a criterion, the so-called diamond
lemma~\cite{Ber78}, based on the notion of $S$-polynomials. Not only do the
$S$-polynomials characterise confluence, but they also serve to construct a
generating set for the syzygies of Gr\"obner bases. From this construction, a
free resolution, named after Schreyer, of the algebra presented by the Gröbner
basis can be obtained by induction. This process requires to extend the notion 
of Gröbner bases, originally introduced for polynomial ideals, to modules over
polynomial algebras~\cite{GG02}.   
\smallskip

Standard bases of formal power series ideals have a syntactic definition which
is analogous to that of Gröbner bases of polynomial ideals. A monomial order is
still used to replace the leading monomial of each generator of the ideal by the
corresponding remainder. Since formal power series are infinite expressions, the
termination is lost in this context. However, using a suitable topological
adaptation of the confluence property, the limit process of replacing leading
monomials using a standard basis still provides a deterministic
result~\cite{Che20, CCM25}. Moreover, using $S$-series, the diamond lemma and
the procedure for constructing a generating set of syzygies of a standard basis
of a formal power series ideal can be adapted~\cite{Bec90a, GH98, Hir64}.   
\smallskip

In this article, we first aim at proving the diamond lemma for standard bases of
modules over algebras of formal power series. We then show how to construct, in
this context, a standard basis for the syzygies of a standard basis. We deduce a
free resolution, analogous to that of Schreyer, of a finitely presented module
over an algebra of formal power series.  

\paragraph{Standard bases of modules over formal power series.}

We mentioned above that the classical approach to rewriting theory over
algebraic structures equipped with vector space operations consists in selecting
one monomial appearing in each relation and simplifying it into the
corresponding remainder. Moreover, the choice of the rewritten monomial is made
using a well-suited notion of monomial order. In this paper, we consider the
free module $\Lambda^p$ of rank $p$ over an algebra of formal power series
$\Lambda$, and monomials are $p$-tuples whose only non-zero entry is a classical 
monomial. A monomial order over the set of monomials
$\mon\left(\Lambda^p\right)$ is a well-order such that the multiplication by any
classical monomial is an non-decreasing map. A standard basis of a sub-module
$L$ of $\Lambda^p$ is a subset $\left\{\gen\right\}$ of $L$ such that the
leading monomial, with respect to the opposite order of the monomial order, of
any element of $L$ is divisible by the leading monomial of some $\vv{s_j}$. To
state a diamond lemma in this context, we still need a notion of standard 
representation of an element of $L$: it is a linear combination of $\vv{s_j}$'s
with coefficients in $\Lambda$ in which there is no cancellation of leading 
monomials. Our first contribution is the following diamond lemma:
\begin{quote}
  {\bf Theorem~\ref{thm:diamond_lemma}.}
  Let $L$ be a sub-module of $\Lambda^p$, $\gen\in L$, and $\monOrd$ be a
  monomial order over~$\mon\left(\Lambda^p\right)$. The following three
  conditions are equivalent:  
  \begin{enumerate}
  \item the set $\left\{\gen\right\}$ is a standard basis of $L$ with respect to 
    $\monOrd$, 
  \item the set $\left\{\gen\right\}$ generates $L$ and for every
    $\vv{f}\in\Lambda^p$, there exists a unique $\vv{r}\in\Lambda^p$ such that 
    \[\vv{f}-\vv{r}\in L\andqquad\supp
    \left(\vv{r}\right)\cap\Span{\lm\left(\gen\right)}=
    \emptyset,\]
  \item every non-zero $\vv{f}\in L$ has a standard representation with respect
    to $\left\{\gen\right\}$ and $\monOrd$. 
  \end{enumerate}
\end{quote}
The notation $\supp\left(\vv{r}\right)$, respectively
$\Span{\lm\left(\gen\right)}$, denotes the support of $\vv{r}$, respectively the
monomials that are divisible by the leading monomial of one $\vv{s_j}$, \ie,
monomials that can be rewritten using the standard basis. Hence, the second item
corresponds precisely to what we mean by deterministic computations: $\vv{r}$ is
the unique representative of $\vv{f}$ modulo $L$ which cannot be rewritten.

\paragraph{Syzygies and resolution of finitely presented modules.}

Our second contribution consists in constructing a free resolution of the module
$\Lambda^p/\Lambda\left(\gen\right)$ with $p$ generators subject to the
relations $\vv{s_j}=\vv{0}$. To do this, we assume that $\left\{\gen\right\}$ is
a standard basis and proceed by induction: we construct a standard basis of the
module of syzygies of $\left\{\gen\right\}$. The induction step is based on the
existence of standard representations. Indeed, given two distinct elements of
the standard basis such that their leading monomials have a common multiple, the
$S$-series of these two elements, \ie, the $\Lambda$-linear combination of these
elements that cancels the least common multiple of the leading monomials,
belongs to the sub-module generated by the standard basis, and therefore has a
standard representation. By collecting the coefficients of this decomposition in 
a tuple of $\Lambda^q$, we obtain a syzygy. Any collection of all the syzygies
constructed in this way is called a Schreyer family. Finally, the original
monomial order over $\mon\left(\Lambda^p\right)$ extends to a monomial order
over $\mon\left(\Lambda^q\right)$, called the induced Schreyer order. Our second
result is stated as follows:   
\begin{quote}
  {\bf Theorem~\ref{thm:std_bases_syz}.} Let $L\subseteq\Lambda^p$ be a
  sub-module, $\monOrd$ be a monomial order over $\mon\left(\Lambda^p\right)$,
  and~$\left\{\gen\right\}~\subset~\Lambda^p$ be a standard basis of $L$ with 
  respect to $\monOrd$. Then, any Schreyer family of $\left\{\gen\right\}$ with
  respect to $\monOrd$ is a standard basis of $\syz^\Lambda\left(\gen\right)$
  with respect to the induced Schreyer order.  
\end{quote}

\section{Conventions and notations}
\label{sec:conv_not}

Throughout the paper, we fix a commutative field $\kk$ and a set $\{\X\}$, where
$d\geq 1$ is an integer. We denote by $\monX$ and by $\Lambda=\kkX$ the
corresponding free commutative monoid and algebra of formal power series, 
respectively. A generic element of $\monX$ is written in the form
$x^\mu=x_1^{\mu_1}\cdots x_d^{\mu_d}$ with multi-exponent
$\mu=\left(\mu_1,\cdots,\mu_d\right)\in\NN^d$. The multiplication in $\monX$ is 
identified to the component-wise addition in $\NN^d$, and $1$ denotes the
neutral element of $\monX$. A {\em monomial order} over $\monX$ is a total order 
$\monOrd$ over $\monX$ such that the following two conditions hold:   
\begin{itemize}
\item for every $x^\mu\in\monX\setminus\{1\}$, we have $1\monOrd x^\mu$, 
\item for every $x^\mu,x^{\mu'},x^\nu\in\monX$, if we have $x^\mu\monOrd
  x^{\mu'}$, then we have $x^{\mu+\nu}\monOrd x^{\mu'+\nu}$. 
\end{itemize}
These two conditions imply that any monomial order makes $\monX$ a {\em
  well-ordered set}, \ie, every non-empty subset of $\monX$ admits a unique
smallest element, see~\cite[Lemma 8.2.3]{BN98}. 
\smallskip

All the modules we will be working with are $\Lambda$-modules; hence, we will
not specify that the coefficient ring is $\Lambda$. For example, "free module"
or "sub-module" means respectively "free $\Lambda$-module" or
"$\Lambda$-sub-module". 
\smallskip

Let $r\geq 1$ be an integer. The free module of rank $r$ is written $\Lambda^r$
and its canonical basis elements are denoted by  $\vve{r,i}$,  for $1\leq i\leq
r$. We introduce the set of monomials of $\Lambda^r$, defined as follows:   
\[\mon\left(\Lambda^r\right)=\left\{x^\mu\vve{r,i}\mid
\mu\in\NN^d,\ 1\leq i\leq r\right\},\]
where $x^\mu\vve{r,i}$ is the element of $\Lambda^r$ whose only non-zero entry
is $x^\mu$, at position $i$, that is
\[x^\mu\vve{r,i}=\left(\underset{i-1}{\underbrace{0,\cdots,0}},
x^\mu,
\underset{r-i}{\underbrace{0,\cdots,0}}\right).\]
A generic element of $\Lambda^r$ is written $\vv{f}$, the coefficient of any
monomial  $x^\mu\vve{r,i}$ in $\vv{f}$ is written $\Span{\vv{f}\mid
  x^\mu\vve{r,i}}$, and we define the {\em support} of $\vv{f}$, written
$\supp\left(\vv{f}\right)$, as follows:  
\[\supp\left(\vv{f}\right)=
\left\{x^\mu\vve{r,i}\mid\Span{\vv{f}\mid x^\mu\vve{r,i}}\neq 0\right\} 
\subseteq\mon\left(\Lambda^r\right).\]
In other words, any element $\vv{f}\in\Lambda^r$ is written as 
\[\vv{f}=
\left(\sum_{\mu\in\NN^d}\Span{\vv{f}\mid x^\mu\vve{r,1}}x^\mu,\cdots,
\sum_{\mu\in\NN^d}\Span{\vv{f}\mid x^\mu\vve{r,r}}x^\mu\right)=
\sum_{x^\mu\vve{r,i}\in\supp\left(\vv{f}\right)}\Span{\vv{f}\mid
  x^\mu\vve{r,i}}x^\mu\vve{r,i}. 
\]
The zero element of $\Lambda^r$ is written $\vv{0}$, it has empty support. Given
a monomial order $\monOrd$ over $\monX$, a {\em monomial order, compatible with
  $\monOrd$}, over $\mon\left(\Lambda^r\right)$ is a total order, still denoted
by $\monOrd$, such that the following two
conditions hold:  
\begin{itemize}
\item for every $\mu,\nu\in\NN^d$ and every $1\leq i\leq r$, if 
  $x^\mu\monOrd x^\nu$, then $x^\mu\vve{r,i}\monOrd x^\nu\vve{r,i}$, 
\item for every $x^\mu\vve{r,i},x^{\mu'}\vve{r,j}\in\mon\left(\Lambda^r\right)$ 
  and every $\nu\in\NN^d$, if $x^\mu\vve{r,i}\monOrd x^{\mu'}\vve{r,j}$, then
  $x^{\mu+\nu}\vve{r,i}\monOrd x^{\mu'+\nu}\vve{r,j}$. 
\end{itemize}
In what follows, all our monomial orders will be compatible with a monomial
order over $\monX$. Hence, we will simply write "monomial order over
$\mon\left(\Lambda^r\right)$" instead of "monomial order over
$\mon\left(\Lambda^r\right)$, compatible with a monomial order over
$\monX$". Since $\monX$ equipped with a monomial order is a well-ordered set,
$\mon\left(\Lambda^r\right)$ equipped with a monomial order is also
well-ordered: the smallest element of a non-empty subset
$S\subseteq\mon\left(\Lambda^r\right)$ is the smallest monomial among all 
monomials of the form $x^{\mu_i}\vve{r,i}$, where $1\leq i\leq r$ is such that
there exists a monomial of the form $x^\mu\vve{r,i}$ that belongs to $S$ and 
then $x^{\mu_i}\vve{r,i}$ is the smallest of these monomials. If $\monOrd$ is a
monomial order over $\mon\left(\Lambda^r\right)$, then for every non-zero
$\vv{f}\in\Lambda^r$, we define the {\em leading monomial} $\lm_\monOrd(f)$ and
{\em leading coefficient} $\lc_\monOrd(f)$ of $\vv{f}$ with respect to $\monOrd$
as follows: 
\[\lm_\monOrd\left(\vv{f}\right)=
\min_\monOrd\left(\supp\left(\vv{f}\right)\right)\in\mon\left(\Lambda^r\right)
\andqquad
\lc_\monOrd\left(\vv{f}\right)=
\Span{\vv{f}\mid\lm_\monOrd\left(\vv{f}\right)}\in\kk.\]
The word "leading" is defined by a $\min$ since it is understood to be the
$\max$ for the opposite order. Given a subset $S\subseteq\Lambda^r$, we let  
\[\lm_\monOrd(S)=\left\{\lm_\monOrd\left(\vv{f}\right)\mid\vv{f}\in
S\setminus\left\{\vv{0}\right\}\right\}.\]
The monoid $\monX$ acts on $\mon\left(\Lambda^r\right)$ by $x^\mu
x^\nu\vve{r,i}=x^{\mu+\nu}\vve{r,i}$, and given any subset 
$S\subseteq\mon\left(\Lambda^r\right)$, we denote by
$\Span{S}\subseteq\mon\left(\Lambda^r\right)$ the orbit of $S$ for this action;
notice that for any sub-module $L\subseteq\Lambda^r$, the set
$\lm_\monOrd\left(L\right)$ is stable for this action. Thanks to this, we are
able to define standard bases for sub-modules of $\Lambda^r$. 

\begin{definition}
  Let $r$ be a strictly positive integer, $L\subseteq\Lambda^r$ be a sub-module
  of $\Lambda^r$, and $\monOrd$ be a monomial order over
  $\mon\left(\Lambda^r\right)$. We say that a subset $S\subseteq L$ is a {\em 
    standard basis} of $L$ with respect to $\monOrd$ if
  $\Span{\lm_\monOrd(S)}=\lm_\monOrd(L)$.  
\end{definition}

Dickson's lemma implies the existence of finite standard bases, a fact we will
now prove. 

\begin{proposition}
  Given a monomial order over $\mon\left(\Lambda^r\right)$, every sub-module of 
  $\Lambda^r$ admits a finite standard basis with respect to this monomial
  order. 
\end{proposition}

\begin{proof}
  Let $L\subseteq\Lambda^r$ be a sub-module of $\Lambda^r$ and $\monOrd$ be a
  monomial order over $\mon\left(\Lambda^r\right)$. For every $1\leq i\leq r$,
  the monomials of the form $x^\mu\vve{r,i}$ that belong to $\lm_\monOrd(L)$ are
  stable under the action of $\monX$. Thus, according to Dickson's lemma, they
  form the orbit of finitely many such monomials. For each $1\leq i\leq r$ and
  each generator of the corresponding orbit, we choose an element of $L$ whose
  leading monomial is this generator. The set of elements thus chosen forms a 
  finite standard basis of $L$ with respect to $\monOrd$.   
\end{proof}

\section{The diamond lemma for modules over formal power series}
\label{sec:diamond_lemma}

The diamond lemma for formal power series ideals is proven in~\cite[Corollaries
  2.2, 2.3]{Bec90a}. In this section, we extend this result to sub-modules of
finite rank free modules. 
\smallskip

Let us fix two positive integers $p$ and $q$, a subset
$S=\left\{\gen\right\}\subset\Lambda^p$, and a monomial order $\monOrd$ over the
set of monomials $\mon\left(\Lambda^p\right)$. Since $\Lambda^p$ is the only
free module we work with in this section, the elements of its canonical basis
are written $\vve{i}$ instead of $\vve{p,i}$, for $1\leq i\leq p$, and since the
monomial order $\monOrd$ is fixed, we will omit it in the notations that refer
to it. Hence, writing  
\[\lm\left(\vv{s_j}\right)=x^{\mu_j}\vve{i_j},\]
for every $1\leq j\leq q$, $\Span{\lm(S)}\subseteq\mon\left(\Lambda^p\right)$ is
the set of monomials of $\Lambda^p$ that are of the form 
\[x^\mu\lm\left(\vv{s_j}\right)=
x^{\mu+\mu_j}\vve{i_j}=
\left(\underset{i_j-1}{\underbrace{0,\cdots,0}},x^{\mu+\mu_j},
\underset{p-i_j}{\underbrace{0,\cdots0}}\right).\]

The following result is the adaptation of~\cite[Proposition 2.1]{Bec90a} to free
modules of arbitrary rank; in essence, it means that the limit of the process of
cancelling leading monomials using $S$ provides a remainder.  
\begin{proposition}\label{prop:nf}
  For any $\vv{f}\in\Lambda^p$, there exist $c_1,\cdots,c_q\in\Lambda$ such
  that 
  \begin{enumerate}
  \item $\supp\left(\vv{f}-\left(c_1\vv{s_1}+\cdots+c_q\vv{s_q}\right)\right)$
    does not intersect $\Span{\lm(S)}$,
  \item $\lm\left(\vv{f}\right)\monOrdLeq\lm\left(c_j\vv{s_j}\right)$, for every
    $j\in\{1,\cdots,q\}$ such that $c_j\neq 0$. 
  \end{enumerate}
\end{proposition}

\begin{proof}
  In the proof, $<$ denotes the order over ordinals. Since
  $\mon\left(\Lambda^p\right)$ equipped with $\monOrd$ is a well-ordered set,
  there exists a unique ordinal $\alpha$ that is isomorphic to this
  well-ordered set, so that we may arrange $\mon\left(\Lambda^p\right)$ into a
  strictly increasing transfinite sequence   
  \[\mon\left(\Lambda^p\right)=
  \left(x^{\mu_\beta}\vve{i_\beta}\right)_{\beta<\alpha}.\]
  We define $c_1,\cdots,c_q$ by constructing their supports and corresponding
  coefficients by transfinite induction. Let $\beta<\alpha$ and assume that for
  every $\gamma<\beta$, we have constructed
  $\lambda_1(\gamma),\cdots,\lambda_q(\gamma)\in\kk$ and
  $\mu_1(\gamma),\cdots,\mu_q(\gamma)\in\NN^d$ in such a way that the following
  three conditions hold:
  \begin{enumerate}
  \item there exists at most one index $j\in\{1,\cdots,q\}$ such that
    $\lambda_j(\gamma)\neq 0$, and if so, we have
    $x^{\mu_\gamma}\vve{i_\gamma}=x^{\mu_j(\gamma)}\lm\left(\vv{s_j}\right)$, 
  \item for every $j\in\{1,\cdots,q\}$ and every $\gamma'<\gamma$, if
    $\lambda_j(\gamma)\neq0$ and $\lambda_j(\gamma')\neq 0$, then
    $x^{\mu_j(\gamma')}\monOrd x^{\mu_j(\gamma)}$,
  \item if we define $\vv{r(\gamma)}$ by
    \[\vv{r(\gamma)}=\vv{f}-
    \sum_{j=1}^q\left(\sum_{\gamma'\leq\gamma}
    \lambda_j(\gamma')x^{\mu_j(\gamma')}\right)\vv{s_j},\]
    then for every $x^\mu\vve{i}\in\supp\left(\vv{r(\gamma)}\right)\cap
    \Span{\lm\left(S\right)}$, we have $x^{\mu_\gamma}\vve{i_\gamma}\monOrd
    x^\mu\vve{i}$.  
  \end{enumerate}
  In order to proceed with the induction step for $\beta$, for every
  $j\in\{1,\cdots,q\}$, we let $\lambda_j(\beta)=0$ and $\mu_j(\beta)\in\NN^d$
  arbitrary, except if $x^{\mu_\beta}\vve{i_\beta}\in\Span{\lm(S)}$, and in this 
  case, we let 
  \[\lambda_j(\beta)=\frac{\Span{\vv{s(\beta)}\mid x^{\mu_\beta}\vve{i_\beta}}}{
    \lc\left(\vv{s_j}\right)}
  \andqquad
  \mu_j(\beta)=\mu,\]
  with one $j\in\{1,\cdots,q\}$ chosen for which there exists
  $\mu\in\NN^d$ so that
  $x^{\mu_\beta}\vve{i_\beta}=x^\mu\lm\left(\vv{s_j}\right)$, and where 
  \[\vv{s(\beta)}=\vv{f}-
  \sum_{j=1}^q\left(\sum_{\gamma<\beta}
  \lambda_j(\gamma)x^{\mu_j(\gamma)}\right)\vv{s_j}.\]
  Let us check that the three induction hypotheses hold for
  $\beta$: 
  \begin{enumerate}
  \item that follows from the definitions of $\lambda_j(\beta)'s$ and
    $\mu_j(\beta)$'s, for $j\in\{1,\cdots,q\}$,
  \item the only case that does not follow immediately from the corresponding
    induction hypothesis is when there exists one $j\in\{1,\cdots,q\}$
    such that $\lambda_j(\beta)\neq 0$ and $\gamma<\beta$ such that
    $\lambda_j(\gamma)\neq 0$; but if that is the case, the inequality
    $x^{\mu_j(\gamma)}\monOrd x^{\mu_j(\beta)}$ is verified because otherwise,
    by definition of a monomial order and according to the first induction
    hypothesis, the inequality $x^{\mu_\gamma}\vve{i_\gamma}\monOrd
    x^{\mu_\beta}\vve{i_\beta}$ would not hold,  
  \item first, consider $\gamma<\beta$ such that
    $x^{\mu_\gamma}\vve{i_\gamma}\in\Span{\lm(S)}$. From the induction
    hypothesis,
    $x^{\mu_\gamma}\vve{i_\gamma}\notin\supp\left(\vv{r(\gamma)}\right)$.
    Moreover, for every $\gamma<\gamma'\leq\beta$ and every
    $j\in\{1,\cdots,q\}$, $\lambda_j(\gamma')x^{\mu_j(\gamma')}\vv{s_j}$ is
    either equal to $\vv{0}$ or has leading monomial
    $x^{\mu_{\gamma'}}\vve{i_{\gamma'}}\succ x^{\mu_\gamma}\vve{i_\gamma}$, so
    that $x^{\mu_\gamma}\vve{i_\gamma}\notin  
    \supp\left(\lambda_j(\gamma')x^{\mu_j(\gamma')}\vv{s_j}\right)$. Hence, we
    get that $x^{\mu_\gamma}\vve{i_\gamma}$ does not belong to the support of 
    \[\vv{r(\beta)}=\vv{r(\gamma)}-
    \sum_{j=1}^q\left(\sum_{\gamma<\gamma'\leq\beta}
    \lambda_j(\gamma')x^{\mu_j(\gamma')}\right)\vv{s_j}.\]
    We have thus proven that every element of
    $\supp\left(\vv{r(\beta)}\right)\cap  
    \Span{\lm(S)}$ is not smaller than $x^{\mu_\beta}\vve{i_\beta}$. To
    conclude, notice that $\lambda_j(\beta)$'s, for $j\in\{1,\cdots,q\}$, are
    chosen so that  $x^{\mu_\beta}\vve{i_\beta}\notin
    \supp\left(\vv{r(\beta)}\right)\cap\Span{\lm(S)}$. 
  \end{enumerate}
  We now define the $c_j$'s based on the second induction hypothesis, namely,
  for every $j\in\{1,\cdots,q\}$, we let 
  \[c_j=\sum_{\beta<\alpha}\lambda_j(\beta)x^{\mu_j(\beta)}.\]
  Then, the first statement of the proposition holds since for every
  $x^{\mu_\gamma}\vve{i_\gamma}\in\Span{\lm(S)}$, the last induction hypothesis
  implies that $x^{\mu_\gamma}\vve{i_\gamma}$ does not belong to
  $\supp\left(\vv{r(\gamma)}\right)$, the first one implies that
  $x^{\mu_\gamma}\vve{i_\gamma}$ does not belong to the support of any element
  of the form $\lambda_j(\beta)x^{\mu_j(\beta)}\vv{s_j}$, for $\beta>\gamma$   
  and $j\in\{1,\cdots,q\}$, and thus $x^{\mu_\gamma}\vve{i_\gamma}$ does not
  belong to the support of  
  \[\vv{f}-\sum_{j=1}^qc_j\vv{s_j}=
  \vv{r(\gamma)}-\sum_{j=1}^q\left(\sum_{\beta>\gamma}
  \lambda_j(\beta)x^{\mu_j(\beta)}\right)\vv{s_j}.\]
  For the second assertion, first notice that from the induction hypotheses, the
  leading monomials of all non-zero $c_j\vv{s_j}$'s, for $1\leq j\leq q$, are 
  pairwise distinct elements of $\Span{\lm(S)}$. Hence, if
  $\lm\left(\vv{f}\right)$ were strictly greater than the smallest 
  monomial among all well-defined $\lm\left(c_j\vv{s_j}\right)$'s, for $1\leq
  j\leq q$, this smallest monomial would be the leading monomial of
  $\vv{f}-\left(c_1\vv{s_1}+\cdots+c_q\vv{s_q}\right)$, contradicting the first
  statement. 
\end{proof}

To state the diamond lemma, we need the following notion of standard
representation.  

\begin{definition}
  A non-zero element $\vv{f}\in\Lambda^p\setminus\left\{\vv{0}\right\}$ is said
  to have a {\em standard representation} with respect to $S$ and $\monOrd$ if 
  there exist $c_1,\cdots,c_q\in\Lambda$ such that the following hold:
  \[\vv{f}=c_1\vv{s_1}+\cdots+c_q\vv{s_q}\andqquad
  \lm\left(\vv{f}\right)=\min\left\{
  \lm\left(c_j\vv{s_j}\right)\mid 1\leq j\leq q, c_j\neq 0\right\}.\]
  In this case, we call $\left(c_1,\cdots,c_q\right)\in\Lambda^q$ a {\em
    standard representation} of $\vv{f}$ with respect to $S$ and $\monOrd$.
\end{definition}

\begin{theorem}\label{thm:diamond_lemma}
  Let $L$ be a sub-module of $\Lambda^p$, $\gen\in L$, and $\monOrd$ be a
  monomial order over $\mon\left(\Lambda^p\right)$. The following three
  conditions are equivalent: 
  \begin{enumerate}
  \item the set $\left\{\gen\right\}$ is a standard basis of $L$ with respect to 
    $\monOrd$, 
  \item the set $\left\{\gen\right\}$ generates $L$ and for every
    $\vv{f}\in\Lambda^p$, there exists a unique $\vv{r}\in\Lambda^p$ such that 
    \[\vv{f}-\vv{r}\in L\andqquad\supp
    \left(\vv{r}\right)\cap\Span{\lm\left(\gen\right)}=
    \emptyset,\]
  \item every non-zero $\vv{f}\in L$ has a standard representation with respect
    to $\left\{\gen\right\}$ and $\monOrd$. 
  \end{enumerate}
\end{theorem}

\begin{proof}
  As above, we let $S=\left\{\vv{s_1},\cdots,\vv{s_q}\right\}$. Moreover, the
  implication $3.\Rightarrow 1.$ being obvious, we only have to prove the
  following two implications. 

  \begin{description}
  \item[$1.\Rightarrow2.$] Fix $\vv{f}\in\Lambda^p$. With the notations of
    Proposition~\ref{prop:nf}, $\vv{r}=\vv{f}-\left(c_1\vv{s_1}+\cdots
    c_q\vv{s_q}\right)$ satisfies the required two conditions. Moreover, if 
    $\vv{f}\in L$, then we have $\vv{r}\in L$; but since
    $\supp\left(\vv{r}\right)\cap\Span{\lm(S)} = \emptyset$ and $S$ 
    is a standard basis of $L$ with respect to $\monOrd$, we get
    $\vv{r}=\vv{0}$. Hence, $L$ is generated by $S$. Finally, if we assume that
    $\vv{r}'$ also satisfies the two conditions of $2.$, then we have
    $\vv{r}-\vv{r}'\in L$ and
    $\supp\left(\vv{r}-\vv{r}'\right)\cap\Span{\lm(S)}=\emptyset$; so using
    again that $S$ is a standard basis of  $L$ with respect to $\monOrd$, we
    conclude that $\vv{r}-\vv{r}'=\vv{0}$.  
  \item[$2.\Rightarrow 3.$] Let $\vv{f}\in L$. With the notations of
    Proposition~\ref{prop:nf}, the element
    $\vv{r}=\vv{f}-\left(c_1\vv{s_1}+\cdots c_q\vv{s_q}\right)$ satisfies the
    condition $2$, and since $L$ is generated by $S$, $\vv{0}$ also satisfies
    this condition. The uniqueness property thus implies
    $\vv{f}=c_1\vv{s_1}+\cdots+c_q\vv{s_q}$. Now, from
    Proposition~\ref{prop:nf}, $\lm\left(\vv{f}\right)$ is not greater than the
    smallest element among $\lm\left(c_j\vv{s_j}\right)$'s, with $1\leq j\leq q$
    such that $c_j$ is non-zero, but it cannot be strictly smaller since
    otherwise, the equation $\vv{f}=c_1\vv{s_1}+\cdots+c_q\vv{s_q}$ would not
    hold. Thus, $\left(c_1,\cdots,c_q\right)$ is a standard representation of
    $\vv{f}$ with respect to $S$ and $\monOrd$. 
  \end{description}
\end{proof}

\section{Standard bases of syzygies and free resolution}
\label{sec:std_basis_syz}

In this section, we construct a standard basis of the module of syzygies of a
standard basis. From this, we deduce how to construct in principle a free
resolution of a finitely presented module.   
\smallskip

The setting is the same as in Section~\ref{sec:diamond_lemma}: $p$ and $q$ are
two positive integers, $S=\left\{\gen\right\}$ is a set of non-zero
elements of $\Lambda^p$, and $\monOrd$ is a monomial order over
$\mon\left(\Lambda^p\right)$. However, since the two free modules $\Lambda^p$ and
$\Lambda^q$ are involved, we keep the notations $\vve{p,i}$ and $\vve{q,j}$,
with $1\leq i\leq p$ and $1\leq j\leq q$, respectively, for the elements of the
corresponding canonical bases. Moreover, we will introduce a monomial order
$\SchOrd$ over $\mon\left(\Lambda^q\right)$. For clarity, we will not indicate
the dependency on the order in the $\lm$ notation. In case of ambiguity, we will
specify whether $\lm$ is applied to an element of $\Lambda^p$ or $\Lambda^q$ in
order to resolve the ambiguity. The notation $\lc$ will always apply to elements
of $\Lambda^p$, hence with respect to $\monOrd$ and never to $\SchOrd$. 
\smallskip

The {\em module of syzygies} of $S$ is defined as the following sub-module of
$\Lambda^q$:    
\[\syz^\Lambda(S)=
\left\{\left(c_1,\cdots,c_q\right)\in\Lambda^q\mid
c_1\vv{s_1}+\cdots+c_q\vv{s_q}=\vv{0}\right\}.\]
An {\em overlapping} of $S$ with respect to $\monOrd$ is a pair of distinct
elements of $S$ whose leading monomials have a common multiple. Formally, it is
a pair $\left(\vv{s_j},\vv{s_{j'}}\right)$, with $1\leq j\neq j'\leq q$, such
that 
\begin{equation}\label{equ:lm_overlapp}
  \lm\left(\vv{s_j}\right)=x^{\mu_j}\vve{p,i}
  \andqquad
  \lm\left(\vv{s_{j'}}\right)=x^{\mu_{j'}}\vve{p,i},
\end{equation}
where $\mu_j,\mu_{j'}\in\monX$ and $1\leq i\leq p$. Moreover, for
$\mu,\mu'\in\NN^d$, we denote the multi-exponent of the least common multiple of
$x^\mu$ and $x^{\mu'}$ in $\monX$ as follows:
\[\lcm\left(\mu,\mu'\right)=
\left(\max\left(\mu_1,\mu'_1\right),\cdots,\max\left(\mu_d,\mu'_d\right)\right)
\in\NN^d.\] 

Now, following~\cite{GG02}, we define the {\em Schreyer order} over 
$\mon\left(\Lambda^q\right)$ induced by $(S,\monOrd)$ as being the order
$\SchOrd$ over $\mon\left(\Lambda^q\right)$ defined by $x^\mu\vve{q,j}\SchOrd   
x^{\mu'}\vve{q,j'}$ if the following holds:   
\[\lm\left(x^\mu\vv{s_j}\right)\monOrd\lm\left(x^{\mu'}\vv{s_{j'}}\right)
\qquad\text{or}\qquad
\left(\lm\left(x^\mu\vv{s_j}\right)=\lm\left(x^{\mu'}\vv{s_{j'}}\right)
\andquad
j<j'\right).\]
Notice that $\SchOrd$ is a monomial over $\mon\left(\Lambda^q\right)$ that is 
compatible with the same monomial order over $\monX$ as the one with which
$\monOrd$ is compatible. 
\smallskip

In the rest of the section, when we write "$S$ is a standard basis with respect 
to $\monOrd$", the underlying module is the sub-module of $\Lambda^p$ generated
by $S$. This is legitimate according to Theorem~\ref{thm:diamond_lemma}.

\begin{lemma}\label{lem:syz}
  Assume that $S$ is a standard basis with respect to $\monOrd$, let
  $\left(\vv{s_j},\vv{s_{j'}}\right)$ be an overlapping of $S$ with respect to
  $\monOrd$, with $1\leq j<j'\leq q$, and let $\mu_j,\mu_{j'}\in\NN^d$ be
  defined by~\eqref{equ:lm_overlapp}. Then, there exists 
  $\left(c_1,\cdots,c_q\right)\in\Lambda^q$ such that the following element
  \[\syz\left(\vv{s_j},\vv{s_{j'}}\right)=
  \frac1{\lc\left(\vv{s_j}\right)}
  x^{\lcm\left(\mu_j,\mu_{j'}\right)-\mu_j}\vve{q,j}-
  \frac1{\lc\left(\vv{s_{j'}}\right)}
  x^{\lcm\left(\mu_j,\mu_{j'}\right)-\mu_{j'}}\vve{q,j'}-
  \sum_{k=1}^qc_k\vve{q,k},\]
  is a non-zero element of $\syz^\Lambda(S)$ and has leading monomial
  $x^{\lcm\left(\mu_j,\mu_{j'}\right)-\mu_j}\vve{q,j}$ with respect to
  $\SchOrd$.   
\end{lemma}

\begin{proof}
 Let us consider the element
  \[\vv{f}=\frac1{\lc\left(\vv{s_j}\right)}
  x^{\lcm\left(\mu_j,\mu_{j'}\right)-\mu_j}\vv{s_j}-
  \frac1{\lc\left(\vv{s_{j'}}\right)}
  x^{\lcm\left(\mu_j,\mu_{j'}\right)-\mu_{j'}}\vv{s_{j'}}.\]
  According to Theorem~\ref{thm:diamond_lemma}, since $\vv{f}$ belongs to the
  sub-module of $\Lambda^p$ generated by $S$, $\vv{f}$ admits a standard
  representation  $\left(c_1,\cdots,c_q\right)\in\Lambda^q$ with respect to $S$
  and $\monOrd$, \ie, we have  
  \[\vv{f}=c_1\vv{s_1}+\cdots+c_q\vv{s_q}\andqquad
  \lm\left(\vv{f}\right)=\min\left\{
  \lm\left(c_k\vv{s_k}\right)\mid 1\leq k\leq q, c_k\neq 0\right\}.\]
  The definition and the above decomposition of $\vv{f}$ yield
  \[\frac1{\lc\left(\vv{s_j}\right)}
  x^{\lcm\left(\mu_j,\mu_{j'}\right)-\mu_j}\vv{s_j}-
  \frac1{\lc\left(\vv{s_{j'}}\right)}
  x^{\lcm\left(\mu_j,\mu_{j'}\right)-\mu_{j'}}\vv{s_{j'}}-
  \sum_{k=1}^qc_k\vv{s_k}=
  \vv{0},\]
  meaning that $\syz\left(\vv{s_j},\vv{s_{j'}}\right)$ belongs to
  $\syz^\Lambda(S)$. Finally, notice that by definition of $\vv{f}$, the
  monomial
  \begin{equation}\label{equ:lem_syz}
    x^{\lcm\left(\mu_j,\mu_{j'}\right)-\mu_j}\lm\left(\vv{s_j}\right)=
    x^{\lcm\left(\mu_j,\mu_{j'}\right)-\mu_{j'}}\lm\left(\vv{s_{j'}}\right),
  \end{equation}
  cancels in $\vv{f}$, so that, if it is well-defined, $\lm\left(\vv{f}\right)$
  is strictly greater than~\eqref{equ:lem_syz} with respect to $\monOrd$. Hence, 
  $\lm\left(c_k\vv{s_k}\right)$ is also strictly greater
  than~\eqref{equ:lem_syz}, for every $1\leq k\leq q$ such that $c_k\neq
  0$. Since $j<j'$, we get that  $\syz\left(\vv{s_j},\vv{s_{j'}}\right)$ is
  non-zero and has leading monomial
  $x^{\lcm\left(\mu_j,\mu_{j'}\right)-\mu_j}\vve{q,j}$ with respect to
  $\SchOrd$, which ends the proof.  
\end{proof}

\begin{definition}
  If $S$ is a standard basis with respect to $\monOrd$, a {\em Schreyer family}
  of $S$ with respect to $\monOrd$ is a set containing one element of the form
  $\syz\left(\vv{s},\vv{s'}\right)$, as in Lemma~\ref{lem:syz}, for each
  overlapping $\left(\vv{s},\vv{s'}\right)$ of $S$ with respect to $\monOrd$. 
\end{definition}

Note that a Schreyer family is finite since its cardinality is equal to the
number of overlapping. 

\begin{theorem}\label{thm:std_bases_syz}
  Let $L\subseteq\Lambda^p$ be a sub-module, $\monOrd$ be a monomial order over
  $\mon\left(\Lambda^p\right)$, and $\left\{\gen\right\}\subset\Lambda^p$ be a
  standard basis of $L$ with respect to $\monOrd$. Then, any Schreyer family of
  $\left\{\gen\right\}$ with respect to $\monOrd$ is a standard basis of
  $\syz^\Lambda\left(\gen\right)$ with respect to the Schreyer order induced by
  $\left(\left\{\gen\right\},\monOrd\right)$.  
\end{theorem}

\begin{proof}
  As above, we write $S=\left\{\gen\right\}$ and denote by $\SchOrd$ the
  Schreyer order induced by $(S,\monOrd)$. Let us define  
  \[I=\left\{(j,j')\in\{1,\cdots,q\}^2\mid\left(s_j,s_{j'}\right)\
  \text{is an overlapping of}\ S\ \text{with respect
    to}\ \monOrd\ \text{and}\ j<j'\right\},\] 
  so that any Schreyer family of $S$ with respect to $\monOrd$ is indexed by
  $I$. We then fix a Schreyer family $\Scal$, written as follows: 
  \[\Scal=\left\{\syz\left(\vv{s_j},\vv{s_{j'}}\right)\mid (j,j')\in I\right\}.
  \]
  According to Lemma~\ref{lem:syz}, we have
  $\Scal\subset\syz^\Lambda(\gen)$. Therefore, Theorem~\ref{thm:diamond_lemma}
  implies that it suffices to show that any syzygy of $S$ admits a standard
  representation with respect to $\Scal$ and $\SchOrd$. Let us fix such a syzygy 
  \[\vv{c}=\left(c_1,\cdots, c_q\right)\in\syz^\Lambda(S).\]
  By applying Proposition~\ref{prop:nf} to $\vv{c}\in\Lambda^q$, $\SchOrd$, and 
  $\Scal$ instead of $\vv{f}$, $\monOrd$, and $S$, respectively, for every
  $(j,j')\in I$, there exists $d_{j,j'}\in\Lambda$ such that writing
  \[\vv{r}=
  \vv{c}-\sum_{(j,j')\in I}d_{j,j'}\syz\left(\vv{s_j},\vv{s_{j'}}\right),\]
  the following holds, where $\lm$ and $\min$ refer to elements of
  $\mon\left(\Lambda^q\right)$: 
  \begin{equation}\label{equ:thm_syz}
    \supp\left(\vv{r}\right)\cap\Span{\lm\left(\Scal\right)}=\emptyset
    \andqquad
    \lm\left(\vv{c}\right)=\min\left(
    \lm\left(d_{j,j'}\syz\left(\vv{s_j},\vv{s_{j'}}\right)\right)\mid
    d_{j,j'}\neq 0\right).
  \end{equation}
  It remains to be proven that $\vv{r}=\vv{0}$ because if it is the case, the
  second condition of~\eqref{equ:thm_syz} implies that
  $\left(d_{j,j'}\right)_{(j,j')\in I}$ is a standard representation
  of $\vv{c}$ with respect to $\Scal$ and $\SchOrd$. Suppose that
  $\vv{r}\neq\vv{0}$ and let $1\leq j\leq q$ be such that the
  $j$th component $r_j$ of $\vv{r}$ is non-zero and
  $\lm\left(r_j\vv{s_j}\right)\in\mon\left(\Lambda^p\right)$ is minimal with
  respect to $\monOrd$ among all monomials of the form
  $\lm\left(r_{j'}\vv{s_{j'}}\right)$ with $r_{j'}\neq 0$. Since $\vv{r}$ is a
  $\Lambda$-linear combination of elements of $\syz^\Lambda(S)$, it belongs to
  this module, \ie, we have 
  \[\sum_{j'=1}^qr_{j'}\vv{s_{j'}}=\vv{0}.\]
  Hence, there exists an index $1\leq j'\leq q$ different from $j$ such that
  $\lm\left(r_j\vv{s_j}\right)=\lm\left(r_{j'}\vv{s_{j'}}\right)$. Let
  $\mu_j,\mu_{j'},\nu_j,\nu_{j'}\in\NN^d$ and $1\leq i\leq p$ such that
  \[\lm\left(\vv{s_j}\right)=x^{\mu_j}\vve{p,i},\qquad
  \lm\left(\vv{s_{j'}}\right)=x^{\mu_{j'}}\vve{p,i},\qquad
  \lm\left(r_j\right)=x^{\nu_j},\andqquad
  \lm\left(r_{j'}\right)=x^{\nu_{j'}},\]
  so that $x^{\nu_j}\vve{q,j}$ and $x^{\nu_{j'}}\vve{q,j'}$ both belong to
  $\supp\left(\vv{r}\right)$. Notice that
  $\lm\left(r_j\vv{s_j}\right)=
  \lm\left(r_{j'}\vv{s_{j'}}\right)\in\mon\left(\Lambda^p\right)$ implies that
  the index $1\leq i\leq p$ is indeed the same for $\lm\left(\vv{s_j}\right)$
  and $\lm\left(\vv{s_{j'}}\right)$, and also that the equality 
  $\mu_j+\nu_j=\mu_{j'}+\nu_{j'}$ holds. By definition of the least common
  multiple, we get that $x^{\lcm\left(\mu_j,\mu_{j'}\right)-\mu_{j'}}$ and
  $x^{\lcm\left(\mu_j,\mu_{j'}\right)-\mu_j}$ divide $x^{\nu_{j'}}$ and
  $x^{\nu_j}$, respectively. From Lemma~\ref{lem:syz}, we get that
  $x^{\nu_j}\vve{q,j}$ or $x^{\nu_{j'}}\vve{q,j'}$ is a multiple of
  $\lm\left(\syz\left(\vv{s_j},\vv{s_{j'}}\right)\right)$ under the action of
  $\monX$ over the set of monomials $\mon\left(\Lambda^q\right)$. Since both
  $x^{\nu_j}\vve{q,j}$ and $x^{\nu_{j'}}\vve{q,j'}$ belong to
  $\supp\left(\vv{r}\right)$, this leads to a contradiction with the first
  condition of~\eqref{equ:thm_syz}, so that $\vv{r}=\vv{0}$, as desired. This
  completes the proof. 
\end{proof}

Finally, one may apply Theorems~\ref{thm:diamond_lemma}
and~\ref{thm:std_bases_syz} to the construction of a free resolution of a
finitely presented module for which the defining relations form a standard
basis. Let $\M=\Lambda^p/\Lambda(S)$, where $\Lambda(S)\subseteq\Lambda^p$
denotes the sub-module generated by $S=\left\{\gen\right\}\subset\Lambda^p$, be
the module with $p$ generators subject to the $q$ relations defined by $S$. Let
$\monOrd$ be a monomial order over $\mon\left(\Lambda^p\right)$ for which $S$ is
a standard basis of $\Lambda(S)$. Then, the free resolution, together with a
family of monomial orders, are constructed by induction as follows. 
\smallskip

For the initial step, we set $q_0=p$, $q_1=q$, $\monOrd_0=\monOrd$, and denote
by 
\[\partial_0:\Lambda^{q_0}\to\M \andqquad
\partial_1:\Lambda^{q_1}\to\Lambda^{q_0},\]
the canonical projection and the homomorphism defined by
$\partial_1\left(\vve{q_1,j}\right)=\vv{s_j}$, for $1\leq j\leq q_1$, 
respectively. We thus have the exact sequence
\[\begin{tikzcd}
\Lambda^{q_1}\ar[r, "\partial_1"] &
\Lambda^{q_0}\ar[r, "\partial_0"] &
\M\ar[r] & 0.
\end{tikzcd}
\]

By induction, assume that for a given integer $n\geq 1$, we have already
constructed an exact sequence  
\begin{equation}\label{equ:partial_res}
  \begin{tikzcd}
    \Lambda^{q_n}\ar[r, "\partial_n"] &
    \Lambda^{q_{n-1}}\ar[r, "\partial_{n-1}"] &
    \cdots\ar[r, "\partial_1"] &
    \Lambda^{q_0}\ar[r, "\partial_0"] &
    \M\ar[r] & 0,
  \end{tikzcd}
\end{equation}
and a monomial order $\monOrd_{n-1}$ over $\mon\left(\Lambda^{q_{n-1}}\right)$
such that the set 
\[\Scal_{n-1}=\left\{\partial_n\left(\vve{q_n,j}\right)\mid 1\leq j\leq
q_n\right\}\subset\Lambda^{q_{n-1}},\]
is a standard basis of the sub-module
$\ker\left(\partial_{n-1}\right)\subseteq\Lambda^{q_{n-1}}$ with respect to
$\monOrd_{n-1}$.
\smallskip

We then proceed by defining the following objects: let
\begin{enumerate}
\item $\Scal_n\subset\Lambda^{q_n}$ be a Schreyer family of $\Scal_{n-1}$ with
  respect to $\monOrd_{n-1}$,
\item $\monOrd_n$ be the Schreyer order over $\mon\left(\Lambda^{q_n}\right)$
  induced by $\left(\Scal_{n-1},\monOrd_{n-1}\right)$,
\item $q_{n+1}\in\NN$ be the cardinality of $\Scal_n$, 
\item $\partial_{q_{n+1}}:\Lambda^{q_{n+1}}\to\Lambda^{q_n}$ be a homomorphism
  such that $\Scal_n=\left\{\partial_{n+1}\left(\vve{q_{n+1},j}\right)\mid
  1\leq j\leq q_{n+1}\right\}$.
\end{enumerate}
From Theorem~\ref{thm:std_bases_syz}, $\Scal_n$ is a standard basis of 
$\ker\left(\partial_n\right)$ with respect to $\monOrd_n$; consequently,
according to Theorem~\ref{thm:diamond_lemma}, it is a generating set for 
$\ker\left(\partial_n\right)$. Hence, the homomorphism $\partial_{n+1}$ defined
above extends the exact sequence~\eqref{equ:partial_res} to degree $n+1$,
namely, we have the following exact sequence: 
\[
  \begin{tikzcd}
    \Lambda^{q_{n+1}}\ar[r, "\partial_{n+1}"] &
    \Lambda^{q_{n}}\ar[r, "\partial_{n}"] &
    \cdots\ar[r, "\partial_1"] &
    \Lambda^{q_0}\ar[r, "\partial_0"] &
    \M\ar[r] & 0.
  \end{tikzcd}
\]
This concludes the induction step.
\bibliography{references}
\end{document}